\documentclass[12pt]{article}
\usepackage{latexsym, amsmath, amsfonts, amsthm, amssymb}
\usepackage{hyperref}
\usepackage{yhmath}
\usepackage{mathtools}
\usepackage{times}
\usepackage{a4wide}
\usepackage{times}
\usepackage{graphicx, tocloft}
\usepackage{caption}
\usepackage{mathrsfs}
\usepackage{tikz}
\usetikzlibrary{matrix,arrows}
\usepackage{ bbold }
\usepackage{comment}
\usepackage[overload]{empheq}

\def \RR {\mathbb R}
\def \NN {\mathbb N}
\def \EE {\mathbb E}

\def \CC {\mathbb C}

\def \eps {\varepsilon}

\def \cF {\mathcal F}

\newcommand{\Var}{\mathrm{Var}}

\DeclarePairedDelimiter\abs{\lvert}{\rvert}%
\DeclarePairedDelimiter\norm{\lVert}{\rVert}%

\newtheorem{theorem}{Theorem}

\newtheorem{lemma}[theorem]{Lemma}

\newtheorem{proposition}[theorem]{Proposition}
\newtheorem{corollary}[theorem]{Corollary}

\theoremstyle{definition}
\newtheorem{remark}[theorem]{Remark}

\def\qed{\hfill $\vcenter{\hrule height .3mm
		\hbox {\vrule width .3mm height 2.1mm \kern 2mm \vrule width .3mm
			height 2.1mm} \hrule height .3mm}$ \bigskip}

\def \Re {{\rm  Re  }}
\def \Im {\rm  Im  }

\title{Polynomial Approximation in \( L^2 \) of the Double Exponential via Complex Analysis}
\author{Pierre Bizeul\textsuperscript{1}
and Boaz Klartag\textsuperscript{2}}

\date{}

\begin{document}

\maketitle

\footnotetext[1]{Department of Mathematics, Technion – Israel Institute of Technology, Haïfa 3200003, Israel. Email: \href{mailto:pbizeul@imj-prg.fr}{pbizeul@imj-prg.fr}. Supported by the European Research Council (ERC) under the European Union’s Horizon 2020 research and innovation programme, grant agreement No. 101001677 “ISOPERIMETRY”.}

\footnotetext[2]{Department of Mathematics, Weizmann Institute of Science, Rehovot 7610001, Israel. Email: \href{mailto:boaz.klartag@weizmann.ac.il}{boaz.klartag@weizmann.ac.il}. Supported by a grant from the Israel Science Foundation (ISF).}

\begin{abstract}
    We study the polynomial approximation problem in $L^2(\mu_1)$ where $\mu_1(dx) = \frac{e^{-\abs{x}}}{2} \ dx$. We show that for any absolutely continuous function $f$,
    $$\sum_{k=1}^{\infty} \log^2(e+k)\langle f,P_k\rangle^2 \ \leq C\left(\int_{\RR} \log^2(e+\abs{x})f^2 d\mu_1 \ + \ \int_{\RR}(f')^2 
    d\mu_1\right)$$
    for some universal constant $C>0$, where $(P_k)_{k\in\NN}$ are the orthonormal polynomials associated with $\mu_1$. This inequality is tight in the sense that $\log^2(e +k)$ on the left hand-side cannot be replaced by $a_k \log^2(e +k)$ with a sequence $a_k \longrightarrow \infty$.   
    When the right hand-side is bounded this inequality implies a logarithmic rate of approximation for $f$, which was previously obtained by Lubinsky. We also obtain some rates of approximation for the product measure $\mu_1^{\otimes d}$ in $\RR^d$ via a tensorization argument. Our proof relies on an explicit formula for the generating function of orthonormal polynomials associated with the weight $\frac{1}{2\cosh(\pi x/2)}$ and some complex analysis.
\end{abstract}

\section*{Introduction}
The study of weighted polynomial approximation on the real line has a long history going back to Bernstein in 1924 \cite{bernstein1924probleme} asking for criterion of density of polynomials in the space of continuous functions. Instead of the $L_\infty$ norm, we restrict our attention to the $L_2$ norm. In that case a well known lemma asserts that if a measure $\mu$ on $\RR$ has exponential moments, that is if $\int e^{a\abs{x}}d\mu <\infty$ for some $a>0$, then polynomials are dense in $L^2(\mu)$ : $\forall f \in L^2(\mu)$,
$$E_n(f,\mu) := \inf_{P\in\mathcal{P}_n}\norm{f-P}_2^2 \to 0.$$
Consider, for any $\alpha\geq0$ the probabilities 
$$\mu_\alpha(dx) = \frac{1}{Z_\alpha}e^{-\abs{x}^\alpha}dx$$
It can be shown that the previous criterion is exact : polynomials are dense if and only if $\alpha\geq1$. One might want to quantify the rate at which the polynomial approximation occurs. The systematic study of this question for the measure $\mu_\alpha$ was carried out by Freud in the seventies (see in particular \cite{freud1977markov}). He established in particular that for regular enough functions
\begin{equation}\label{eq26}
    E_n(f,\mu_\alpha) \lesssim n^{1-\alpha}\int_\RR(f')^2d\mu_\alpha
\end{equation}
for $\alpha\geq2$, where here and in the rest of this paper, we write, for real valued functions or sequences, $a\lesssim b$ when there exists a constant $C>0$ such that $a\leq Cb$. Inequality \eqref{eq26} was later extended to $\alpha>1$ by Lubinsky and Levin \cite{levin1987canonical}. The case $\alpha=1$ is critical, it can be actually  shown that the space 
$$ \left \{f,\ \int fd\mu_1 = 0, \quad \int (f')^2d\mu_1\leq 1 \right\}$$
is not compact in $L^2(\mu_1)$,  we refer to \cite{bakry2014analysis} for details. Consequently, a result in the form of \eqref{eq26} is not possible with any given rate. Nevertheless, after a first paper in the $L_1$ setting by Freud, Giroux and Rahman \cite{freud1978approximation}, it was shown by Lubinsky in 2006 \cite{lubinsky2006jackson} that 
\begin{equation}\label{eqLubinsky}
    E_n(f,\mu_1) \ \lesssim \ \frac{\int_{x\leq\sqrt{n}} (f')^2d\mu_1}{\log^2(n+1)} \ + \ \int_{x\geq c\sqrt{n}} f^2 d\mu_1.
\end{equation}
In the $L^2$ setting, it is desirable to obtain more precise information. Given a measure $\mu$, we write for a function $f\in L^2(\mu)$
$$f = \sum_{k\geq0}f_kP_k$$
where $(P_k)_{k\in\NN}$ are the orthonormal polynomials associated to $\mu$ with $\deg(P_k) = k$ and $$ f_k = \langle f\ ,\ P_k\rangle_{L^2(\mu)} $$ are the coefficient of $f$ in this basis. Our main theorem is the following

\begin{theorem}\label{thm:main}
Let $\mu_1$ be the double sided exponential measure, with density $\frac{e^{-\abs{x}}}{2}$ on the real line. For any absolutely continuous function $f$
\begin{equation}
    \sum_{k=1}^{\infty} \log^2(e+k)f_k^2 \ \lesssim \int_{\RR} \log^2(e+\abs{x})f^2 d\mu_1 \ + \ \int_{\RR}(f')^2 
    d\mu_1\label{eq59}
\end{equation}
and
\begin{equation}
        \sum_{k=1}^{\infty} \log^2(e+k)f_k^2 \lesssim \int_\RR\log^2(e+\abs{x})(f')^2d\mu_1.\label{eq60}
\end{equation}
In particular,
\begin{equation}\label{eq_163}
     E_n(\mu_1,f) \lesssim \frac{\int_{\RR} \log^2(e+\abs{x})f^2 d\mu_1 \ + \ \int_{\RR}(f')^2 d\mu_1.}{\log^2(n)}
\end{equation}
Furthermore, the result is tight in a strong sense : for any positive increasing sequence $(a_k)_{k\geq1}$ such that $\lim_{k\to\infty} a_k = +\infty$, there exists a function $f_a$ such that $$\int \log^2(e+\abs{x})(f_a)^2 
 + (f_a')^2 d\mu_1 \leq 1,$$ but 
$$ \sum_{k=1}^{\infty} a_k\log^2(e+k)(f_a)_k^2 = +\infty$$
\end{theorem}

We conjecture that inequality (\ref{eq59}) may be reversed, and in fact for any absolutely continuous function $f$,
$$ \sum_{k=1}^{\infty} \log^2(e+k)f_k^2 \ \gtrsim \int_{\RR} \log^2(e+\abs{x})f^2 d\mu_1 \ + \ \int_{\RR}(f')^2 d\mu_1. $$

\medskip Comparing to \eqref{eqLubinsky}, we get a stronger control on the coefficients $f_k$, showing that the sequence $(\log^2(e+k)f_k^2)_{k\in\NN}$ is in $\ell^1$ rather than in weak $\ell^1$. The cost is a slightly bigger right hand side. Indeed, 
\begin{align*}
    \int_{\lvert x \rvert \geq c\sqrt{n}} f^2 \, d\mu_1 
    &\leq 
    \frac{1}{\log^2\big(1 + c\sqrt{n}\big)}
    \int_{\lvert x \rvert \geq c\sqrt{n}} f^2 \log^2\big(e + \lvert x \rvert\big) \, d\mu_1 \\[10pt]
    &\lesssim 
    \frac{\int f^2 \log^2\big(1 + \lvert x \rvert\big) \, d\mu_1}{\log^2(n)}.
\end{align*}
An important benefit of \eqref{eq60} over \eqref{eqLubinsky} is that it allows for tensorization in higher dimension.

\begin{theorem}{(Tensorization)} Let $\mu_i$ be a probability measure on $\RR$ for $i=1,\ldots,d$. Assume that  $\mu_i$ satisfies
    $$\sum_{\alpha\geq 1} \varphi_i(\alpha) f_\alpha^2 \leq \int_\RR (f')^2w_i(x) \ d\mu_i $$
    for some positive functions $(\varphi_i)_{1\leq i \leq d}$ where $f_\alpha$ are the coefficients in the orthonormal polynomials of $\mu_i$. 
    Denote $$ \mu = \mu_1\otimes\ldots\otimes\mu_d. $$
    Then, it holds for $\mu$ that 
    $$\sum_{|\alpha| \geq 1} \varphi(\alpha) f_\alpha^2 \leq \int_{\RR^d} \sum_{i=1}^d w_i(x_i) (\partial_i f)^2 d\mu$$
    where $\varphi(\alpha) = \sum \varphi_i(\alpha_i)$ and we set $\varphi_i(0) = 0$ and $|\alpha| = \sum_i |\alpha_i|$. In particular, setting $\phi(n) = \inf_{\abs{\alpha} = n}\varphi(\alpha)$, we get :
    $$E_n(\mu, f) \leq \frac{1}{\phi(n)} \int_{\RR^d} \sum_{i=1}^d w_i(x_i) (\partial_i f)^2 d\mu.$$
    As a consequence, if $f$ is $1$-Lipschitz,
    $$E_n(\mu, f) \leq  \frac{1}{\phi(n)} \EE \max w_i(X_i)$$
\end{theorem}
\begin{proof}
    For notational simplicity, we prove the result for $d=2$. Let $L = (L_{\alpha_1})_{\alpha_1\geq0}$ and $M = (M_{\alpha_2})_{\alpha_2\geq0}$ be the orthonormal polynomials for $\mu_1$ and $\mu_2$. Then $L\otimes M$ is the basis of orthonormal polynomials of $\mu$.
    
    Let $\alpha = (\alpha_1,\alpha_2)$ be a double index, and remark that 
    $$f_\alpha = \langle f \ , \ L_{\alpha_1}\otimes M_{\alpha_2} \rangle_{\mu} = \langle \langle f \ , \ L_{\alpha_1} \rangle_{\mu_1} \ , \ M_{\alpha_2} \rangle_{\mu_2} = \left(g_{\alpha_1}\right)_{\alpha_2} $$
    where $g_{\alpha_1} := \langle f \ , \ L_{\alpha_1} \rangle_{\mu_1}$. We apply the hypothesis for $\mu_2$ and find that 
    $$\sum_{\alpha_2 \geq 0} \varphi_2(\alpha_2)\left(g_{\alpha_1}\right)_{\alpha_2}^2 \leq \int_\RR (g_{\alpha_1}')^2(x_2)w_2(x_2) \ d\mu_2(x_2) $$
    Now, $g_{\alpha_1}' = \langle \partial_2 f \ , \ L_{\alpha_1} \rangle_{\mu_1}$, and by Plancherel formula,
    $$ \sum_{\alpha_1 \geq 0} (g_{\alpha_1}')^2 = \int (\partial_2 f)^2 d\mu_1 $$
    Injecting this into the previous equation yields:
    $$\sum_{\alpha_1 \geq 0}\sum_{\alpha_2 \geq 0} \varphi_2(\alpha_2)f_{(\alpha_1,\alpha_2)}^2 \leq \int_{\RR^2} (\partial_2 f)^2w_2(x_2) \ d\mu$$

    By symmetry we obtain a similar inequality with reverse indices, and summing them yields the result.
\end{proof}
\begin{corollary} Let $\mu_1^{\otimes d} = \mu_1 \otimes \dots\otimes \mu_1$ be the product exponential measure on $\RR^d$. Then for any regular enough $f$,
and $n \geq 2$,
$$E_n(f,\mu_1^{\otimes d}) \ \lesssim \ \frac{1}{\log^2(n)}\int_{\RR^d} \sum_{i=1}^d \log^2(e+\abs{x_i}) (\partial_i f)^2d\mu_1^{\otimes d}$$
And if $f$ is $1$-Lipschitz,
$$E_n(f,\mu_1^{\otimes d}) \ \lesssim \ \frac{\EE\max\log^2(X_i)}{\log^2(n)} \lesssim \frac{\log^2\log(d)}{\log^2(n)}. $$
\end{corollary}
It is interesting to compare those results with what happens for the half sided exponential 
$$d\tilde{\mu}_1(x) = e^{-x}\mathbb{1}_{x\geq0}\ dx.$$
The orthonormal polynomials for $\mu_1$ are the Laguerre polynomials $(L_k)_{k\in\NN}$ defined by
$$L_k(x) = \frac{e^{x}}{k!}\frac{d^k}{dx^k}\left(x^k e^{-x}\right).$$
Their derivatives $(L_k')_{k\geq1}$ form an orthogonal system for the weight $xe^{-x}\mathbb{1}_{x\geq0}$. For any $k,m\geq1$
\begin{equation}\label{eq218}
    \int_{\RR}L_k'L_m' \ xe^{-x}\ dx = k\delta_{km}
\end{equation}
Let $f\in L^2(\tilde{\mu}_1)$, we write
$$f=\sum_{k\geq0}f_kL_k$$
so that, if $f'\in L^2(xe^{-x}\mathbb{1}_{x\geq0})$, we have
$$f' = \sum_{k\geq1}f_kL_k'.$$
From \eqref{eq218} we deduce that
\begin{equation}
    \int (f')^2xe^{-x}\mathbb{1}_{x\geq0} \ dx = \sum_{k\geq1} kf_k^2
\end{equation}
So that,
$$E_n(f,\tilde{\mu}_1) =\sum_{k\geq n+1}f_k^2 \leq \frac{\int_\RR x(f')^2\ d\tilde{\mu}_1}{n+1}$$
which is a much faster rate than \eqref{eq_163}. In particular if $f$ is $1$-Lipschitz, one gets a logarithmic rate for $E_n(f,\mu_1)$ but a linear rate for $E_n(f,\tilde{\mu}_1)$. Lastly, the tensorization for $\tilde{\mu}_1^{\otimes d}$ writes
$$E_n(f,\tilde{\mu}_1^{\otimes d}) \ \leq \ \frac{1}{n+1}\int_{\RR^d} \sum_{i=1}^d x_i(\partial_i f)^2d\mu_1^{\otimes d}$$
And if $f$ is $1$-Lipschitz,
$$E_n(f,\mu_1^{\otimes d}) \ \leq \ \frac{\EE\max(\abs{X_i})}{n+1} \lesssim \frac{\log(d)}{n}.$$

\begin{remark}
    Let us reformulate the phenomena in dimension one. Under the two sided exponential, if one allows to approximate Lipschitz functions not only by polynomials, but by piecewise polynomials with exactly two pieces, one on each half line, one goes from a logarithmic to a linear rate of approximation.
\end{remark}
\section{Preliminaries}\label{section_preliminaries}
The proof of Theorem \ref{thm:main} requires some preparation. The first step is to smooth the density a bit. Remark that 
$$\frac{1}{2\cosh(x)} \leq \exp(-\abs{x}) \leq \frac{1}{2\cosh(x)}. $$
Thus we may as well study the approximation problem for the probability
$$ d\nu(x) = \frac{1}{2\cosh(\pi x /2)}.$$
The $\pi/2$ factor is a normalization the appeal of will be clear later. Before all, we give a rigorous justification of this point,
that it suffices to prove Theorem \ref{thm:main} with $\mu_1$ replaced by $\nu$.

\begin{lemma}\label{lem_158}
    Let $(a_k)_{k\geq 0}$ and $(b_k)_{k\geq 0}$ be two positive sequences, $c, C > 0$. For any $n\geq 0$, we define $A_n =\sum_{k\geq n} a_k$ and $B_n = \sum_{k\geq n}b_k$ and we assume that for all $n\in\NN, \ c B_n \leq A_n\leq C B_n$. Let $(\varphi_k)_{k\in \NN}$ be an increasing positive sequence. If 
    $$\sum_{k\in\NN}b_k\varphi_k \leq \infty$$
    then $\sum a_k\varphi_k$ converge and
        $$c\sum_{k\in\NN} b_k\varphi_k \leq \sum_{k\in\NN} a_k\varphi_k \leq C\sum_{k\in\NN} b_k\varphi_k$$
\end{lemma}

Lemma \ref{lem_158} is proven by a standard argument using summation by parts, which we omit.
Now let $\mu$ and $\nu$ are two measures satisfying $c\nu \leq \mu \leq C\nu$. Let $f$ be a function in $L^2(\mu) = L^2(\nu)$. We denote by $f_k^\mu$ and $f_k^\nu$ the coefficients of $f$ in the corresponding basis of orthogonal polynomials. We have for any $n\geq0$
\begin{equation}\label{eq_189}
    c^2\sum_{k\geq n}(f_k^{\nu})^2 = c^2\inf_{deg(P)\leq n}\norm{f-P}_{L^2(\nu)}^2 \leq \norm{f-P}_{L^2(\mu)}^2 \leq \sum_{k\geq n}(f_k^{\mu})^2 \leq C^2\sum_{k\geq n}(f_k^{\mu})^2 
\end{equation}
Thus by Lemma \ref{lem_158} for any positive increasing sequence $(\varphi_k)_{k\geq0}$ we get that 
$$\sum_{k\in\NN}\varphi_k (f_k^{\nu})^2 <\infty \quad \Longleftrightarrow \quad \sum_{k\in\NN}\varphi_k (f_k^{\mu})^2 <\infty$$
in which case we have $$c^2\sum_{k\in\NN}\varphi_k (f_k^{\nu})^2 \leq \sum_{k\in\NN}\varphi_k (f_k^{\mu})^2 \leq C^2\sum_{k\in\NN}\varphi_k (f_k^{\mu})^2.$$
In other words, up to a constant $4$ we may as well establish Theorem \ref{thm:main} for $\nu$ instead of $\mu_1$. As it turns out, the family of orthogonal polynomials for $\nu$
is well known.

\subsection{The Meixner-Pollaczek polynomials.}
Fix a positive integer $\ell$ and consider the $\ell$-fold convolution 
$$
\nu_\ell = \underbrace{\nu * \ldots * \nu}_{\ell \ \text{times}}.
$$
Thus for instance $$ d \nu_2(x) = \frac{x}{2 \sinh(\pi x / 2)} dx. $$ 
While we will not make any usage of the explicit formula for the density of $\nu_{\ell}$, we remark that the density of $\nu_{\ell}$ equals 
$$ \frac{1}{2 \pi} \left|\Gamma \left(\frac{\ell + i x}{2}   \right) \right|^2 = \frac{|\Gamma(\ell/2)|^2}{2 \pi}  \prod_{n=0}^{\infty} \left[ 1 + \left( \frac{x}{\ell + 2n} \right)^2 \right]^{-1}, $$
see Szeg\"o \cite[Appendix]{szeg1939orthogonal}. For $\ell \geq 1, x,s \in \CC$ with $|s| < 1$ we define
$$ G_\ell(x,s) = \frac{e^{x \arctan s}}{(1 + s^2)^{\ell/2}}. $$
By considering the Taylor expansion  in the $s$ variable around $0$, which converges for $|s| < 1$, we see that
\begin{equation}  G_{\ell}(x,s) = \sum_{k=0}^{\infty} P_k^{(\ell)}(x) s^k \label{eq_505} \end{equation}
where $P_k^{(\ell)}$ is a polynomial of degree $k$ whose leading coefficient is $1/k!$.

\begin{lemma} Let $\ell \geq 1$. Then the polynomials $P_0^{(\ell)}, P_1^{(\ell)}, \ldots$ are orthogonal polynomials relative to $\nu_{\ell}$ with
$$ \left \| P_{k}^{(\ell)} \right \|_{L^2(\nu_\ell)}^2 =  {k + \ell - 1 \choose k}. $$
\label{lem_1427}
\end{lemma}

\begin{proof} We will use the trigonometric identity
\begin{equation} \cos (\arctan s + \arctan t) \sqrt{1 + s^2} \sqrt{1 + t^2}  = 1-st, \label{eq_1720}
\end{equation}
as well as the following standard integral (see e.g. \cite{araaya2004meixner}): for $a \in \CC$ with $|a| < \pi/2$,
\begin{equation}  \int_{-\infty}^{\infty} e^{\alpha x} d \nu(x) = \frac{1}{\cos \alpha}. \label{eq_1709} \end{equation}
It follows from (\ref{eq_1709}) that
$$ \int_{-\infty}^{\infty} e^{\alpha x} d \nu_{\ell}(x) = \frac{1}{\cos^{\ell} \alpha}. $$
Consequently,
$$
\int_{-\infty}^{\infty} G_{\ell}(x,s) G_{\ell}(x,t) d \nu_{\ell}(x) = \frac{1}{\cos^{\ell} (\arctan s + \arctan t) (1 + s^2)^{\ell/2} (1 + t^2)^{\ell/2}} = \frac{1}{(1-st)^{\ell}}.
$$
Expanding in power series in $s$ and $t$ we thus get
$$ \sum_{k_1, k_2=0}^{\infty}  \langle P_{k_1}^{(\ell)}, P_{k_2}^{(\ell)}\rangle_{L^2(\nu_{\ell})}  s^{k_1} t^{k_2} = \sum_{k=0}^{\infty} {k + \ell - 1 \choose k} (st)^{k},
 $$
and the conclusion follows by comparing coefficients.
\end{proof}

The polynomials $P_0^{(\ell)}(2x),P_1^{(\ell)}(2x),\ldots$ are referred to as the Pollaczek polynomials in Szeg\"o \cite[Appendix, Section 2]{szeg1939orthogonal}
corresponding to $\alpha = \pi/2$ and $\ell = 2 \lambda$. 
In Araaya \cite{araaya2004meixner} these polynomials are reffered to as the Meixner-Pollaczek (MP) polynomials. 

\begin{lemma} Let $f$ be a polynomial. Then
$$	\sum_{k=1}^{\infty} {k + \ell - 1 \choose \ell} \left \langle f, \frac{ P_k^{(\ell)} }{\| P_k^{(\ell)} \|_{L^2(\nu_{\ell})}} \right \rangle_{L^2(\nu_{\ell})}^2  = \frac{1}{4} \int_{-\infty}^{\infty} |f(x+i) -  f(x-i)|^2 d \nu_{\ell+1}(x). $$
\label{lem_1704}
\end{lemma}

\begin{proof} For two functions $f,g: \CC \rightarrow \CC$ we write
$$  \langle f, g \rangle_{M^2(\nu_{\ell})} = \frac{1}{4} \int_{-\infty}^{\infty} \left(f(x+i) - f(x-i) \right) \left(\overline{g(x+i) - g(x-i)} \right) d \nu_\ell(x) $$
whenever the integrals converge. This defines a positive semi-definite inner product on a space of functions that includes all polynomials.
We claim that the polynomials $P_1^{(\ell)},P_2^{(\ell)},\ldots$ are orthogonal polynomials also with respect to the scalar product $\langle \cdot, \cdot \rangle_{M^2(\nu_{\ell+1})}$
with 
$$ \langle P_k, P_k \rangle_{M^2(\nu_{\ell+1})} =  {k + \ell - 1 \choose \ell} $$
and with $\langle P_0, P_0 \rangle_2 = 0$. In order to prove our claim, we view $s$ and $t$ as fixed real numbers, and compute $\langle G_{\ell}(x, s), G_{\ell}(x, t) \rangle_{M^2(\nu_{\ell+1})}$ with respect to the $x$-variable. Since
$$ \frac{G_{\ell}(x + i, s) - G_{\ell}(x-i, s)}{2} = \frac{e^{(x+i) \arctan s} - e^{(x-i) \arctan s} }{2 (1+s^2)^{\ell/2}} =  \frac{is \cdot e^{x \arctan s} }{(1+s^2)^{(\ell+1)/2}}.
$$
we have
\begin{align*}  \langle G_{\ell}(x, s), G_{\ell}(x, t) \rangle_{M^2(\nu_{\ell+1})} & = \int_{-\infty}^{\infty} \frac{is \cdot e^{x \arctan s} }{(1+s^2)^{(\ell+1)/2}} \overline{\frac{it \cdot e^{x \arctan t} }{(1+t^2)^{(\ell+1)/2}}} d \nu_{\ell+1}(x) \\ & = s t \int_{-\infty}^{\infty} G_{\ell+1}(x, s) G_{\ell+1}(x, t) d \nu_{\ell+1}(x). \end{align*}
This integral equals $(1 - st)^{-(\ell+1)}$, by the computation of the previous lemma. Therefore, by using (\ref{eq_505}) and expanding in power series in $s$ and $t$,
$$ \sum_{k_1, k_2=0}^{\infty}  \langle P_{k_1}, P_{k_2} \rangle_{M_2(\nu_{\ell+1})}  s^{k_1} t^{k_2} = (st) \sum_{k=0}^{\infty} {k + \ell  \choose k}  (st)^{k} = \sum_{k=1}^{\infty} {k + \ell - 1 \choose k-1}  (st)^{k}.
 $$ This proves the above claim. In order to conclude the proof of the lemma, write $f = \sum_k a_k P_k^{(\ell)}$ with $a_k = \langle f, P_k^{(\ell)} \rangle_{L^2(\nu_{\ell})} / {k + \ell - 1 \choose k}$ and compute $\langle f, f \rangle_{M^2(\nu_{\ell+1})}$.
\end{proof}

For simplicity, we now specialize to our case of interest, $\ell = 1$, though the general case may be analyzed along similar lines. 
When $\ell = 1$ and $P_k = P_k^{(1)}$, Lemma \ref{lem_1704} implies that for any polynomial $f$,
\begin{equation}\label{eq_331}	\sum_{k=1}^{\infty} k \langle f, P_k \rangle_{L^2(\nu)}^2 = \frac{1}{4} \int_{-\infty}^{\infty} |f(x+i) -  f(x-i)|^2 d \nu_2(x).  \end{equation}
where we abbreviated $\nu=\nu_1$. In the next section, we explain how to extend \eqref{eq_331} to holomorphic functions with controlled growth in the strip $\{\abs{Im z} \leq 1\}$ (see Lemma \ref{lem_less_general}). Unfortunately, the right hand side can be made arbitrarily large even for well behaved functions $f$. Indeed, let $\lambda\geq1$ and consider the function
$$F_\lambda(x) = \frac{1}{\sqrt{\lambda}}F_1(\lambda x) = \frac{1}{\sqrt{\lambda}}e^{-\lambda^2 x^2/2}$$
and remark that 
$$\int_\RR \abs{F'_\lambda(x)}^2\log^2{(e+\abs{x})}e^{-\abs{x}} \leq\int \abs{F_1'(\lambda x)}^2 \lambda \ dx = \int_\RR F_1'(x)^2 \ dx = \sqrt{2\pi} $$
where we used that $\log^2(e+\abs{x})e^{-x} \leq 1$ for all real $x$. Furthermore $F_\lambda$ is holomorphic and bounded on any horizontal strip. Let us compute the right hand side of \eqref{eq_331}. First,
$$e^{-\lambda^2(x+i)^2/2} - e^{-\lambda^2(x-i)^2/2} = -2ie^{-\lambda^2(x^2-1)/2}\sin(\lambda^2x)$$
So that,
\begin{align*}
    \frac{1}{4}\int_\RR \abs{F_\lambda(x+i)-F_\lambda(x-i)}^2d\nu_2(x) &= \frac{e^{\lambda^2}}{\lambda}\int_\RR e^{-\lambda^2x^2}\sin^2(\lambda^2x)d\nu_2(x) \\
    &\geq \frac{e^{\lambda^2}}{\lambda}\int_{1/4\lambda^2}^{1/2\lambda^2}e^{-\lambda^2x^2}\sin^2(\lambda^2x)d\nu_2(x) \\
    &\gtrsim \frac{e^{\lambda^2}}{\lambda^{3}}
\end{align*}
which can be made arbitrarily large. 

On the other hand a crude upper-bound is 
$$\frac{1}{4}\int_\RR \abs{F_\lambda(x+i)-F_\lambda(x-i)}^2d\nu_2(x) \lesssim \  e^{\lambda^2}.$$
So that by Lemma \ref{lem_less_general} below, which verifies that  \eqref{eq_331} is applicable here, for any $n\geq1$
$$E_n(\nu,F_\lambda) \lesssim \frac{e^{\lambda^2}}{n}.$$
But we also have, for any $n\geq1$,
\begin{align*}
    E_n(\nu,F_\lambda) &\leq \int_\RR \abs{F_\lambda(x)-0}^2d\nu\\
    &\leq \int_\RR F_\lambda^2(x) \ dx \\
    & =\frac{1}{\lambda}\int_\RR e^{-\lambda^2x^2}\ dx \lesssim\frac{1}{\lambda^2}.
\end{align*}
Putting both bounds together, we finally arrive at :

$$E_n(\nu,F_\lambda)\lesssim\min\left(\frac{e^{\lambda^2}}{n}, \frac{1}{\lambda^2}\right) \lesssim\frac{1}{\log(n)}.$$
This suggests looking at 
$$\sum_{k\geq1}\log^\alpha(e+k)f_k^2$$
for some $\alpha > 0$, instead of 
$$\sum_{k\geq1}kf_k^2.$$
This is made possible in the next section, where we develop a machinery to analyze more general weighted sums of the coefficients $f_k^2$. The functions $F_\lambda$ will be used again in Section \ref{subsec_tightness}, with a more refined analysis, to prove the tightness of Theorem \ref{thm:main}.

\begin{remark}
    Instead of $F_\lambda$ we could have looked at
    $$\tilde{F}_\lambda = \frac{1}{\lambda}e^{-\lambda^2x^2/2} = \frac{1}{\sqrt{\lambda}}F_\lambda$$
    which is $\sqrt{2\pi}$ Lipschitz and arrive at the same conclusions. More generally, given $m\geq 1$ one can construct 
    $$F_\lambda^m = \frac{1}{\lambda^{k_m}}e^{-\lambda^2x^2/C_m}$$
    where $k_m\geq1$ and $C_m>0$ are chosen so that the first $m$ derivatives are bounded :
    $$\norm{(F_\lambda^m)^{(k)}}_\infty \leq 1 \quad \quad k=0,1,\ldots l$$
    and still find that 
    $$\sum_{k}k(F_\lambda^m)_k^2 \xrightarrow[\lambda \to +\infty]{} +\infty$$
\end{remark}
\section{A general formula}
Our goal now is to extend Lemma \ref{lem_1704} in two directions : for a broader class of functions as well as a more general weighting of the coefficients.

We say that $f\in \mathcal{F}_1$ if $f \in L^2(\nu)$ and $f$ admits an holomorphic extension to a neighborhood of the strip $\{\abs{Im(z)} \leq 1\}$,
with controlled growth at infinity: $$ \abs{f(R+iv)} = O(e^{\alpha \abs{R}}) $$ for some $\alpha<\frac{\pi}{4}$, uniformly in $\abs{v}\leq 1$ as $\abs{R}\mapsto \infty$. We say that $f\in\mathcal{F}_2$ if both $f\in\mathcal{F}_1$ and $f'\in\mathcal{F}_1$.

For a positive function $\varphi : (0,1)\mapsto \RR^+$ we define, for all $k\geq0$
$$\Gamma_\varphi(k) = 2k\int_0^1 (1-\eps)^{2k}\varphi(\eps)d\eps. $$
Thus $\Gamma_\varphi(k) = \frac{2k}{2k+1}\EE \varphi(Z_k)$, where $Z_k \sim (2k+1)(1-x)^{2k}1_{0\leq x\leq1}\ dx$. For an integrable function $g$ we also define its Fourier transform :
$$\forall v\in\RR, \quad \hat{g}(v) = \int_\RR g(x)e^{ixv}dx.$$
\noindent The goal of this section is to establish the two following formulas:
\begin{lemma}\label{lem_less_general}
Let $f\in\mathcal{F}_1$, for all $k\geq0$, denote by $f_k = \langle f,P_k \rangle$ the coefficients of $f$ in the basis of MP polynomials. The following formula holds:
\begin{equation}\label{eq_k}
\sum_{k=1}^\infty kf_k^2 = \frac{1}{4} \int_{-\infty}^{\infty} |f(x+i) -  f(x-i)|^2 d \nu_2(x).
\end{equation}
\end{lemma}

\begin{lemma}\label{lem_general} Let $f\in\mathcal{F}_1$, for all $k\geq0$, denote by $f_k = \langle f,P_k \rangle$ the coefficients of $f$ in the basis of MP polynomials.
Denote  $\Delta_f(x) = f(x+i)-f(x-i)$, and for real $u,x$, $$ K_u(x) = \frac{xe^{ux}}{\sinh(\pi x/2)}. $$ Let $\varphi$ be a positive function on the interval $(0,1)$
and let $\Phi(t) = \int_0^t \varphi(u)\ du$. Then
\begin{equation}
\sum_{k=1}^\infty \Gamma_\varphi(k)f_k^2 \leq \frac{1}{8\pi} \int_{-\pi/4}^{\pi/4} \int_{\RR} \ \abs{\widehat{K_u\Delta_f}}^2 \ \Phi\left(\left(\pi -4\abs{u}\right)e^{-2\abs{v}}\right) dvdu
\end{equation}
and
\begin{equation}\label{eq_260}
\sum_{k=1}^\infty \Gamma_\varphi(k)f_k^2 \geq \frac{1}{8\pi} \int_{-\pi/4}^{\pi/4} \int_{\RR} \ \abs{\widehat{K_u\Delta_f}}^2 \ \Phi\left(\frac{1}{2\pi}\left(\pi -4\abs{u}\right)e^{-2\abs{v}}\right) dvdu
\end{equation}
where the convention $\varphi=0$ outside of $(0,1)$ is taken to extend $\Phi$ to $(0,\pi).$

\end{lemma}
We begin with proving Lemma \ref{lem_less_general}. Abbreviate $G(x,s) = G_1(x,s)$ from (\ref{eq_505}). For $f \in L^2(\nu)$ we write
$$ G_f(s) = \sum_{k=0}^{\infty} \langle f, P_k \rangle_{L^2(\nu)} s^k = \int_{-\infty}^{\infty} G(x,s) f(x) d\nu(x). $$
Thus, for instance, if $f(x) = 1$ then $G_f(s) = 1$ and if $f(x) = x$ then 
$G_f(s) = s$. In the general case,
\begin{equation} G_f'(s) = \sum_{k=1}^{\infty} k \langle f, P_k \rangle_{L^2(\nu)} s^{k-1}. \label{eq_929} \end{equation}
When $f \in L^2(\nu)$, the power series defining $G_f$ converges to a holomorphic function in the unit disc $\{ s \in \CC \, ; \, |s| < 1 \}$. 
\begin{lemma} We have
$$
\frac{1}{\pi} \int_{|s| < 1} |G_f'(s)|^2 \, d\lambda(s) = \sum_{k=1}^{\infty} k \langle f, P_k \rangle_{L^2(\nu)}^2,
$$
where $\lambda$ stands for the Lebesgue measure. \label{lem_1052}
\end{lemma}
\begin{proof}
    The proof follows from polar integration after remarking that for any $k,l\in \NN$, $$\int_{C(0,r)}z^k\bar{z}^l\ dz=r^{k+l}\int_{0}^{2\pi} e^{i(k-l)\theta}\ d\theta = 2\pi r^{k+l}\delta_{kl}.$$ 
where $C(0,r)$ is the complex circle of center $0$ and radius $r$.
\end{proof}

From now on we assume that $f\in\mathcal{F}_1$. We work with the standard branch of $\log(z)$ and $\sqrt{z}$ with domain $\CC \setminus (-\infty, 0]$. Recall that $\log (1+x) = \sum_n \frac{(-1)^{n+1} x^n}{n}$ and
$$ \arctan s = \sum_{n=0}^{\infty} (-1)^n \frac{s^{2n+1}}{2n+1} = \frac{i}{2} \log \frac{1 - i s}{1 + i s} $$
The M\"obius transformation $s \mapsto \frac{1 - i s}{1 + is}$ maps the unit disc $\{ s \in \CC \, ; \, |s| < 1 \}$ to the right half plane. Thus  $s \mapsto \arctan(s)$ maps the unit disc to the strip $|Re(z)| \leq \pi/4$. Changing variables 
$$ z = \arctan(s) $$
we observe that we need no Jacobian in the formula
\begin{equation}  \int_{|s| < 1} \left|\frac{d G_f(s)}{d s} \right|^2 \, d\lambda(s) = \int_{|Re(z)| < \pi/4} \left|\frac{d G_f(s(z))}{d z} \right|^2 \, d\lambda(z) = \int_{|Re(z)| < \pi/4} \left|H_f'(z) \right|^2 \, d\lambda(z)
  \label{eq_1712} \end{equation}
where $s(z) = \tan z$ and $H_f(z) = G_f(s(z))$. Setting $H(x,z) = G(x,s(z))$ we have $$ H(x, z) = \cos(z) e^{xz} $$
 and
$$ H_f(z) = \cos(z) \int_{-\infty}^{\infty} e^{xz} f(x) d \nu(x). $$

\begin{proof}[Proof of Lemma \ref{lem_less_general}.] 
Let $a,b \in \CC$ be such that $\tilde{f}(z) = f(z) + az + b$ vanishes at $z = i$ and at $z = -i$. The meromorphic function
	$$ F(z) = \frac{\tilde{f}(z)}{4 \cosh(\pi z / 2)} $$
is in fact holomorporphic in the strip $|Im(z)| < 3$; the denominator has a simple zero at $z = \pm i$, yet the numerator vanishes there too. We have
$$ H_f(z) = H_{\tilde{f}}(z) - H_{ax+b}(z) = H_{\tilde{f}}(z) - \left( a \tan(z) + b \right) $$
where
$$ H_{\tilde{f}}(z) = \int_{-\infty}^{\infty} \left( e^{(x + i)z} + e^{(x - i) z} \right) F(x) dx  = \int_{-\infty}^{\infty} e^{x z} \left( F(x - i ) + F(x + i ) \right) dx.
 $$
The last passage follows by changing the contour of integration from $\RR$ to $\RR \pm i $, which is allowed as $|Re(z)| < \pi/4$ and the holomorphic function $e^{x z} F(x)$ decays to $0$ at infinity on the strip $|Im(x)| \leq 1$. Consequently,
\begin{equation}  H_{f}'(z) = \int_{-\infty}^{\infty} e^{x z} x \left(  F(x - i ) + F(x + i ) \right)   dx - \frac{a}{\cos^2 z} = \int_{-\infty}^{\infty} e^{x z} \psi_f(x)  dx 
	\label{eq_1024} \end{equation}
for $$ \psi_f(x) = x ( F(x-i) +  F(x + i) ) - \frac{ax}{2 \sinh(\pi x/2)} = \frac{ix}{4 \sinh(\pi x / 2)} \left(  f(x-i) -  f(x+i)  \right). $$ 
The function $\psi_f$ decays exponentially at infinity, and formula (\ref{eq_1024}) presents
$ H_f'$ as the Laplace transform of $\psi_f$. It is a standard theme in Paley-Wiener theory to represent holomorphic functions 
on the strip as a Laplace transform and use the Parseval identity as follows: 
\begin{align*}  \int_{|\Re(z)| < \pi/4} & \left|H_f'(z) \right|^2 \, d\lambda(z)
 = \int_{-\pi/4}^{\pi/4} \left( \int_{-\infty}^{\infty}  \left|\int_{-\infty}^{\infty} e^{x (u + i v)} \psi_f(x)  dx \right|^2   dv \right) du \\ & =  	
	\int_{-\pi/4}^{\pi/4} \int_{-\infty}^{\infty}  \left| \widehat{e^{xu} \psi_f(x)}(v)  \right|^2   dv du = 2 \pi 	\int_{-\pi/4}^{\pi/4} \int_{-\infty}^{\infty}   e^{2xu} |\psi_f(x)|^2 dx  du \\ & = 2 \pi \int_{-\infty}^{\infty}   \frac{\sinh(\pi x / 2)}{x} |\psi_f(x)|^2 dx.
\end{align*}
Therefore
$$ \frac{1}{\pi} \int_{|\Re(z)| < \pi/4}  \left|H_f'(z) \right|^2 \, d\lambda(z) = \frac{1}{8} \int_{-\infty}^{\infty}   \frac{x}{\sinh(\pi x / 2)} |f(x-i) -  f(x+i)|^2 dx $$
and the conclusion follows from Lemma \ref{lem_1052} and formula (\ref{eq_1712}).
\end{proof}
We will prove Lemma \ref{lem_general} by slightly modifying the previous proof. We begin with an elementary lemma about analytic functions on the disk.
\begin{proposition}\label{prop_234} Let $G=\sum_{k=0}^\infty a_kz^k$ be an analytic function on the unit disk. Then, for any positive weight $\varphi$,
$$\int_0^1\int_{D(0,1-\eps)} \abs{G'}^2\varphi(\eps) \ d\eps = \frac{\pi}{2}\sum_{k=1}^\infty \Gamma_\varphi(k)\ a_k^2 $$
\end{proposition}
\begin{proof}
Remember that 
$$\Gamma_\varphi(k) = 2k\int_0^1 (1-\eps)^{2k}\varphi(\eps)d\eps. $$
The proof again follows from polar integration and the formula $$\int_{C(0,r)}z^k\bar{z}^l\ dz=r^{k+l}\int_{0}^{2\pi} e^{i(k-l)\theta}\ d\theta = 2\pi r^{k+l}\delta_{kl}.$$ 
valid for any $k,l\in \NN$.
\end{proof}
Thus we need to understand the integral of $G_{f}'$ 
on a disk of radius $r=1-\eps$. We will use the same change of variable as before so we need to understand the image of disk of radius $r$ by $\arctan$. For the reader convenience, we include a graphical representation of $A(1-\varepsilon) = \arctan(D(0,1-\varepsilon))$.

\begin{figure}[h]
    \centering
    \includegraphics[width=0.72\textwidth]{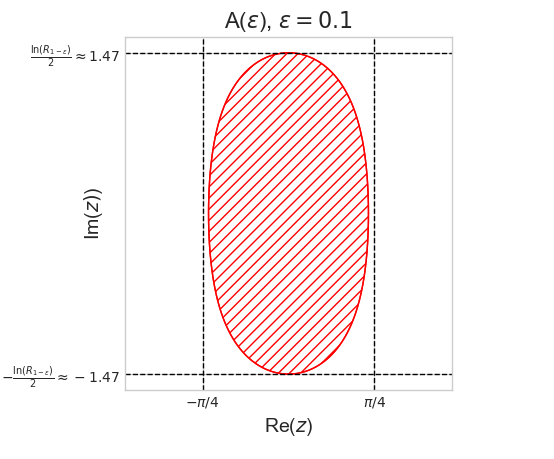} 
    \caption{Illustration of the region $A(\varepsilon) = \arctan(D(0,1-\varepsilon))$ for $\varepsilon=0.1$. The dashed vertical lines correspond to $Re(z) = \pm \pi/4$ while the horizontal lines correspond to $Im(z) =\pm \frac{1}{2}\log R_{1-\epsilon} = \pm \frac{1}{2}\log R_{0,1-\epsilon}$ where $R_{\theta,r}$ is defined in Lemma \ref{lem_467} below.} 
    \label{fig:arctan_disk}
\end{figure}

\begin{lemma}\label{lem_467} Let $0<r<1$ be a positive number, denote $A(r) = \arctan(D(0,r))$ the image of the disk. Let $C_r = \frac{1+r^2}{1-r^2}$. Then

$$A(r) =\bigcup_{2\abs{\theta} \leq \arccos(\frac{1}{C_r})} \{\theta\}\times\left[  -\frac{1}{2}\log\left(R_{\theta,r}\right) \ , \ \frac{1}{2}\log\left(R_{\theta,r}\right)\right].$$
where $R_{\theta,r} = C_r\cos(2\theta) + \sqrt{C_r^2\cos^2(2\theta)-1}$ and we identified $\CC \sim \RR^2$.
\end{lemma}
\begin{proof}
    The proof is rather straightforward, recall that 
    \begin{equation}\label{eq_arctandef}
        \arctan(s) = \frac{i}{2} \log \frac{1 - i s}{1 + i s}
    \end{equation}
    Let $g$ be the Mobiüs transform $g(s) = \frac{1 - i s}{1 + i s}$, we aim to determine $B(r)=g(D(0,r))$, which is a disk. It is easily checked that  $\overline{g(a+ib)} = g(-a + ib)$, for $a+ib \in D(0,1)$. Thus $B(r)$ is symmetric about the horizontal axis, and we only need to look for the minimal and maximal real values of $g$ on $D(0,r)$, which again by symmetry are attained at $\pm r\ i$ and have values 
    $$\frac{1-r}{1+r} \quad \text{and} \quad \frac{1+r}{1-r},$$
    respectively. We denote $K=\frac{1+r}{1-r}$. We have shown that $B(r)$ is the disk of center $(\frac{K+1/K}{2},0) = (C_K,0)$ and radius $R_K = \frac{K-1/K}{2}$.
    We seek to understand $I_r(\theta) = B(r) \cap \RR_+ e^{i\theta}$ for $\abs{\theta}\leq \pi/2$, thus we solve for $\lambda$ the equation
    $$(\lambda\cos(\theta) - C_K)^2 + \lambda^2\sin^2(\theta) = R_K^2$$
    which simplifies to
    $$\lambda^2 -2C_K\cos(\theta)\lambda +1 = 0$$
    where we used that $C_K^2 - R_K^2 = (C_K+R_K)(C_K-R_K) = 1.$
    
    Thus when $C_K\cos(\theta)\geq1$, we find solutions
    $$\lambda_0 = C_K\cos(\theta) - \sqrt{C_K^2\cos^2(\theta) - 1} \quad\quad \lambda_1 = C_K\cos(\theta) + \sqrt{C_K^2\cos^2(\theta) - 1}.$$
    It remains to notice that $\lambda_0\lambda_1=1$, which can be seen from the coefficients of the polynomial equation, and that  by \eqref{eq_arctandef},
    \begin{align*}
        A(r) &= \left\{\left(-\theta/2 \ ,\ \frac{1}{2}\log(r)\right), \quad re^{i\theta}\in B(r)\right\} \\&
        = \bigcup_{|\theta| \leq \frac{1}{2} \arccos (1/C_K)} \{\theta\}\times \frac{\log(I_r(2\theta))}{2}.
    \end{align*}
The lemma follows by noting that $C_K = (1+r^2) / (1 - r^2)$.
\end{proof}
We are now in position to prove Lemma \ref{lem_general}. 
\begin{proof}[Proof of Lemma \ref{lem_general}]

Remember the notations, for reals $u,x$ :
$$\Delta_f(x) = f(x+i)-f(x-i),$$ 
and
$$ K_u(x) = \frac{xe^{ux}}{\sinh(\pi x/2)}. $$
Using Proposition \ref{prop_234} and mimicking the proof of Lemma \ref{lem_less_general}, we get 
\begin{align*}
    \sum_{k=1}^\infty \Gamma_\varphi(k)\ a_k^2 &=\frac{2}{\pi} \int_0^1\int_{D(0,1-\eps)} \abs{G_f'}^2 d \lambda(z) \ \varphi(\eps) \ d\eps\\
    & = \frac{2}{\pi}\int_0^1 \int_{A(1-\eps)} \left|H_f'(z) \right|^2 \, d\lambda(z) \ \varphi(\eps) \ d\eps\\
    & = \frac{2}{\pi}\int_0^1 \int_{A(1-\eps)}\left|\int_{-\infty}^{\infty} \frac{e^{x (u + i v)}}{4\sinh(\pi x/2)}\left(f(x+i)-f(x-i)\right)  dx \right|^2 \, dudv \ \varphi(\eps) \ d\eps\\
    & = \frac{1}{8\pi}\int_0^1 \int_{A(1-\eps)}\left| \widehat{K_u\Delta_f}\right|^2 \, dudv \ \varphi(\eps) \ d\eps\\
    & \leq \frac{1}{8\pi} \int_0^1 \int_{\abs{u}\leq \frac{1}{2}\arccos(\frac{1}{C_{1-\eps}})} \int_{\abs{v}\leq \frac{1}{2}\log\left(R_{u,1-\eps}\right)}\abs{\widehat{K_u\Delta_f}}^2 \ \varphi(\eps) \ dvdud\eps \\
    & = \frac{1}{8\pi} \int_{-\pi/4}^{\pi/4} \int_{\RR} \ \abs{\widehat{K_u\Delta_f}}^2 \ \Phi\left(a(u,v)\right) dvdu\\
\end{align*}
where we used the formula \eqref{eq_1024} for $H_f'$ and Fubini in the last line where $a(u,v)$ is the greatest real such that:
\begin{align}[left={\empheqlbrace}]
&\cos(2u) \geq \frac{1}{C_{1-a(u,v)}}\label{eq:condA} \\
\bigskip
&R_{u,1-a(u,v)} \geq e^{2\abs{v}}\label{eq:condB}
\end{align}
We shall now prove that 
\begin{equation}\label{eq:abound}
    \frac{1}{2\pi}\left(\pi - 4\abs{u}\right)e^{-2\abs{v}} \leq a(u,v) \leq \left(\pi - 4\abs{u}\right)e^{-2\abs{v}}.
\end{equation}
It is easily checked that the function
$$\iota:\eps \mapsto \frac{1}{C_{1-\eps}} = \frac{1-(1-\eps)^2}{1+(1-\eps)^2} = \frac{\eps(2-\eps)}{2-2\eps + \eps^2}$$
is an increasing bijection of $[0,1]$. Furthermore, for $0<\eps<1$, it can be checked that
$$\eps \leq \iota(\eps) \leq 2\eps.$$
Thus $j=\iota^{-1}$ is increasing and satisfies
$$\eps/2 \leq j(\eps) \leq \eps.$$
From the definition of $R_{\theta,r}$ (see Lemma \ref{lem_467}) we have for any $0<r<1$ and any $\theta \in [0,2\pi]$
\begin{equation}\label{eq_536}
    C_r\cos(2\theta) \leq R_{\theta,r} \leq 2C_r\cos(2\theta).
\end{equation}
Combining \eqref{eq_536} and \eqref{eq:condB}, we get:
\begin{align*}
    &2C_{1-a(u,v)}\cos(2u) \geq e^{2\abs{v}}\\
    &\implies \frac{1}{C_{1-a(u,v)}} \leq2\cos(2u)e^{-2\abs{v}}\\
    &\implies \iota(a(u,v)) \leq2\cos(2u)e^{-2\abs{v}}\\
    &\implies a(u,v) \leq 2\cos(2u)e^{-2\abs{v}}\\
    &\implies a(u,v) \leq (\pi-4\abs{u})e^{-2\abs{v}}
\end{align*}
where we used that $j(\eps)\leq \eps$ in the third line, and lastly that $\cos(2u)=\sin(\pi/2-2\abs{u})\leq \pi/2-2\abs{u}$. 

For the lower-bound, set $$\tilde{a}(u,v)=\frac{1}{2\pi}\left(\pi - 4\abs{u}\right)e^{-2\abs{v}}.$$ We want to show that 
$$a(u,v) \geq \tilde{a}(u,v).$$
By definition of $a(_,v)$, we need to show that $\tilde{a}(u,v)$ satisfies \eqref{eq:condA} and \eqref{eq:condB}. First observe that, by concavity
$$\forall x\in[0,\pi/2] \quad\quad  \sin(x) \geq \frac{2x}{\pi}.$$
Thus for any $\abs{u}\leq \pi/4$
$$\cos(2u) = \sin(\pi/2-2\abs{u}) \geq 1-\frac{4\abs{u}}{\pi} \geq 2\tilde{a}(u,v) \geq\frac{1}{C_{1-\tilde{a}(u,v)}}$$
where in the last inequality we used that $j(\varepsilon)\geq \varepsilon/2$. Thus, \eqref{eq:condA} is satisfied. 

For \eqref{eq:condB}, using the left hand side of \eqref{eq_536} :
\begin{align*}
    \frac{1}{R_{u,1-\tilde{a}(u,v)}} &\leq \frac{1}{C_{1-\tilde{a}(u,v)}\cos(2u)}\\
    & = \frac{1}{\iota(\tilde{a}(u,v))\cos(2u)}\\
    &\leq 2\frac{\tilde{a}(u,v)}{\cos(2u)}\\
    &\leq 2\pi\frac{\tilde{a}(u,v)}{\pi-4\abs{u}}\\
    &\leq e^{-2\abs{v}}
\end{align*}
where we again used the inequality  $\cos(2u)\geq 1-\frac{4\abs{u}}{\pi}$. We have proved \eqref{eq:abound}. The lemma is proved since $\Phi$ is increasing.
\end{proof}
We will need  to compute $\Gamma_\varphi$ for certain reasonable functions $\varphi$. This is the purpose of the next lemma.
\begin{lemma}\label{lemma_gammacompute}
     Let $\varphi$ be a decreasing function then for any $k\geq1$,
     \begin{equation}\label{eq_351}
         \frac{\varphi(1/k)}{32} + \frac{k}{2}\Phi(1/2k) \, \leq \, \Gamma_\varphi(k) \, \leq \, 2k\Phi(1/k) + \varphi(1/k).
     \end{equation}
\end{lemma}
\begin{proof}
Remember that 
$$\Gamma_\varphi(k) = 2k\int_0^1 (1-\eps)^{2k}\varphi(\eps)d\eps. $$
Proving \eqref{eq_351} is straightforward. For the upper-bound 
the right hand side follows from separating the integral on two integrals on $[0,1/k]$ and $[1/k,1]$ respectively. 

For the left-hand side, we lower bound the integral on $[0,1]$ by the integral on $[0,1/2k]$ plus the integral on $[1/2k,3k/4]$.
It remains to observe that on $[0,1/2k]$ :
$$\left(1-\frac{1}{2k}\right)^{2k} \geq \left(1-\frac{1}{2}\right)^{2} =1/4,$$
while on the second interval, $[1/2k,3k/4]$, which has length $k/4$,
$$\left(1-\frac{3}{4k}\right)^{2k} \geq \left(1-\frac{3}{4}\right)^{2} =1/16.$$
\end{proof}
Thus for instance if $0<\beta<1$, $\varphi(\eps) = \eps^{-\beta}$, for all $k\geq1$
$$\Gamma_\varphi(k) \simeq \frac{k^\beta}{1-\beta}.$$
And if $\varphi(\eps) = \log^2(\eps)$, for all $k\geq 1$,
        $$\Gamma_\varphi(k) \simeq \log^2(e+k).$$   

\section{Proof of Theorem \ref{thm:main}}
We begin with proving the upper-bound part of Theorem \ref{thm:main} when $f\in\mathcal{F}_2$. 
\subsection{Upper-bound}
According to Lemma \ref{lemma_gammacompute}, we choose $\varphi(\eps) = \log^2(\eps)$, so that $\Gamma_\varphi(k) \simeq \log^2(e+k)$ and, for any $0\leq x\leq \pi$,
$$\Phi(x) = x\left(\log^2\left(\frac{x}{e}\right)+1\right) \leq 2x\log^2\left(\frac{x}{e\pi}\right).$$
Indeed, let $0\leq x\leq \pi$ we have $\log^2\left(\frac{x}{e\pi}\right) \geq \log^2(1/e) = 1$, and thus it is enough to check that $\log^2\left(\frac{x}{e\pi}\right) - \log^2(x/e) \geq 0$, which is elementary. By Lemma \ref{lem_general},
\begin{align*}
    &\sum_{k=1}^\infty \Gamma_\varphi(k)f_k^2 \leq \frac{1}{8\pi} \int_{-\pi/4}^{\pi/4} \int_{\RR} \ \abs{\widehat{K_u\Delta_f}}^2 \ \Phi\left(\left(\pi -4\abs{u}\right)e^{-2\abs{v}}\right) dvdu \\
    &\lesssim\int_{-\pi/4}^{\pi/4}\int_{-\infty}^{\infty} \abs{\widehat{K_u\Delta_f}}^2 e^{-2\abs{v}}\left(\pi-4\abs{u}\right)\log^2\left(\frac{1}{e}\left(1-\frac{4\abs{u}}{\pi}\right)e^{-2\abs{v}}\right) \ dvdu\\
    &\leq 2\int_{-\pi/4}^{\pi/4}\int_{-\infty}^{\infty} \abs{\widehat{K_u\Delta_f}}^2 e^{-2\abs{v}}\left(\pi-4\abs{u}\right)\left[\log^2\left( \frac{1}{e}\left(1-\frac{4\abs{u}}{\pi}\right) \right) \ + 4\abs{v}^2\right]\ dvdu\\
    &\leq 8\int_{-\pi/4}^{\pi/4}\int_{-\infty}^{\infty} \abs{\widehat{\frac{d}{dx}K_u\Delta_f}\ e^{-\abs{v}}}^2 \ \left(\pi-4\abs{u}\right)\ dvdu \ + 2\int_{-\pi/4}^{\pi/4}\int_{-\infty}^{\infty} \abs{\widehat{K_u\Delta_{f}} \ e^{-\abs{v}}}^2 \ \psi(u) \ dvdu \\
    & \leq 16 \int_{-\pi/4}^{\pi/4}\int_{-\infty}^{\infty} \abs{\widehat{K_u\Delta_{f'}} \ e^{-\abs{v}}}^2 \ \left(\pi-4\abs{u}\right) \ dvdu \ + \ 16\int_{-\pi/4}^{\pi/4}\int_{-\infty}^{\infty} \abs{\widehat{K'_u\Delta_f} \ e^{-\abs{v}}}^2 \ \left(\pi-4\abs{u}\right)\ dvdu \\
    & + 2\int_{-\pi/4}^{\pi/4}\int_{-\infty}^{\infty} \abs{\widehat{K_u\Delta_{f}} \ e^{-\abs{v}}}^2 \ \psi(u) \ dvdu \\
    &= 16 T_1 \quad + \quad 16 T_2 \quad + \quad 2T_3 \\
\end{align*}
where $\psi(u) =\left(\pi-4\abs{u}\right) \log^2\left(\frac{1}{e}\left(1-\frac{4}{\pi}\abs{u}\right)\right)$. Now we remark that for a function $g \in \mathcal{F}_1$ that is real on the real line,
\begin{align*}
    \int_{\RR}\abs{\widehat{K_u\Delta_g}(v) e^{-\abs{v}}}^2dv &=\int_\RR\abs{\left(\widehat{K_ug(x+i)}(v) - \widehat{K_ug(x-i)}(v)\right) e^{-\abs{v}}}^2dv \\
    &= \int_\RR\abs{\left(\widehat{K_u(x-i)g}(v)e^v - \widehat{K_u(x+i)g}(v)e^{-v}\right) e^{-\abs{v}}}^2dv \\
    &\lesssim \int_\RR\abs{\widehat{K_u(x-i)g}(v)}^2 \ dv \ +\ \int_\RR\abs{\widehat{K_u(x+i)g}(v)}^2 \ dv \\
    &\lesssim \int_\RR\left(\abs{K_u(x+i)}^2 +\abs{K_u(x-i)}^2\right)g^2(x) \ dx \\
    &\lesssim \int_\RR \frac{e^{2xu}\left(1+x^2\right)}{\cosh(\pi x)} \ g^2(x) \ dx
\end{align*}
where we used that for any $h\in L^1$, holomorphic near a strip $\{ |\Im(z)| \leq 1 \}$ such that $h(x+i)\in L^1$
$$\widehat{x\mapsto h(x+i)}(v) = \hat{h}(v)e^v,$$
the Parseval identity, and the observation that for any $u$,
$$K_u(x+i) = \overline{K_u(x-i)} = \frac{e^{iu}e^{xu}(x+i)}{i\cosh{\pi x /2}}.$$
We thus get
$$T_1 \lesssim \int_\RR (f')^2 \ \frac{1+x^2}{\cosh(\pi x)}\left(\int_{-\pi/4}^{\pi/4} e^{2xu} \left(\pi-4\abs{u}\right) \ du\right) \ dx$$
And
$$T_3 \lesssim \int_\RR f^2 \ \frac{1+x^2}{\cosh(\pi x)}\left(\int_{-\pi/4}^{\pi/4} e^{2xu} \psi(u) \ du\right) \ dx.$$
From an entirely analogous argument and after computing $K_u'$ we also get :
$$T_2 \lesssim \int_\RR f^2 \ \frac{1+x^2}{\cosh(\pi x)}\left(\int_{-\pi/4}^{\pi/4} e^{2xu} \left(\pi-4\abs{u}\right) \ du\right) \ dx$$
and we are left with computing the integrals in $u$. Let $\alpha\in\RR$, integration by parts shows that 
$$I(\alpha) = \int_{-1}^{1} e^{\alpha t}(1-\abs{t}) \ dt = \int_{0}^{1} (e^{\alpha t}+e^{-\alpha t})(1-t) \ dt = \frac{2(\cosh(\alpha)-1)}{\alpha^2}\leq \frac{4\cosh(\alpha)}{1+\alpha^2} $$
where we used $\cosh \alpha \leq 1 + \alpha^2 $ for $\abs{\alpha}\leq 1$. 
We deduce the value of the first integral :
\begin{equation*}
    \int_{-\pi/4}^{\pi/4} e^{2xu} \left(\pi-4\abs{u}\right) \ du = \frac{\pi^2}{4}I(\pi x/2) \lesssim \frac{\cosh(\pi x/2)}{1+x^2}.
\end{equation*}
We move to computing the second integral. We change variables to the interval $[-1,1]$, writing 
\begin{align*}
    \int_{-\pi/4}^{\pi/4} e^{2xu} \psi(u) \ du & = \frac{\pi}{4}\int_{-1}^{1} e^{\frac{\pi x}{2}t}\left(\pi-\pi\abs{t}\right)\log^2\left(\frac{1-\abs{t}}{e}\right) \ dt.
\end{align*}
The function $\Psi : x\mapsto x\log^2(x/e)$ is concave on $[0,1]$, and thus by Jensen inequality
\begin{align*}
    J(\alpha) &= \int_{-1}^{1}e^{\alpha t}(1-\abs{t})\log^2\left(\frac{1-\abs{t}}{e}\right) \ dt \\
    &\leq 2 \int_{0}^1 \cosh(\alpha t) \ dt \ \Psi\left(\frac{\int_{0}^1 \cosh(\alpha t)(1-\abs{t}) \ dt}{\int_{0}^1 \cosh(\alpha t)dt}  \right) \\
    &= I(\alpha) \ \log^2\left(\frac{\alpha I(\alpha)}{ 2e \sinh(\alpha)}\right)\\
    & = \frac{2(\cosh(\alpha)-1)}{\alpha^2} \log^2\left(\frac{\cosh(\alpha)-1}{e \alpha\sinh(\alpha)}\right) \\
    &\leq \frac{4\cosh(\alpha)}{1+\alpha^2}\log^2\left(2e+2e\abs{\alpha}\right)
\end{align*}
where we used that, for any $\alpha \in \mathbb{R}$, 
\begin{equation}
\frac{1}{2(1 + |\alpha|)} \leq \frac{\cosh(\alpha) - 1}{\alpha \sinh(\alpha)} \leq \frac{1}{2}
\label{eq_1549} \end{equation}
which can be checked by direct differentiation , and, as before,
\begin{equation}
\frac{\cosh(\alpha) - 1}{\alpha} \leq \frac{2 \cosh(\alpha)}{1 + \alpha^2}.
\label{eq_1550}
\end{equation}
See the Appendix for the elementary proofs of (\ref{eq_1549}) and (\ref{eq_1550}).
Plugging this into the upper bounds of $T_1$, $T_2$, and $T_3$, with $\alpha_x = \frac{\pi x}{2}$, we get in the end
\begin{align*}
     \sum_{k=1}^\infty \log(e+k)^2 f_k^2 \ &\lesssim \ \int_\RR \left(f^2\log^2(2+\pi \abs{x})+(f')^2\right)\frac{\cosh(\pi x /2)}{\cosh(\pi x)}  \ dx\\
     & \lesssim \int_\RR f^2\log^2(e + \abs{x})\ d\nu(x) \ +   \ \int_\RR (f')^2\ d\nu(x)
\end{align*}
We proved \eqref{eq59} and it remains to establish \eqref{eq60}. Notice that the left hand side of $\eqref{eq59}$ is translation invariant, so that optimizing over translations we get :
$$\sum_{k=1}^\infty \log^2(e+k) f_k^2\ \lesssim \ \inf_{a\in\RR}\int_\RR (f-a)^2\log^2(e + \abs{x})\ d\nu(x) \ +   \ \int_\RR (f')^2\ d\nu(x)$$
and we want to prove that 
$$ \sum_{k=1}^{\infty} \log^2(e+k)f_k^2 \ \lesssim \ \int_\RR\log^2(e+\abs{x})(f')^2d\nu.$$
Thus it is enough to prove that 
\begin{equation}\label{eq_672}
    \inf_{a\in\RR}\int_\RR (f-a)^2\log^2(e + \abs{x})\ d\nu(x) \lesssim \int_\RR\log^2(e+\abs{x})(f')^2d\nu.
\end{equation}
This is an instance of a Poincaré inequality. We say that a measure $\mu$ satisfies a Poincaré inequality if for regular enough $f$,
$$\Var_\mu(f) = \inf_{a\in\RR}\int (f-a)^2 \ d\mu \leq C\int (f')^2d\mu$$
for some constant $C>0$. We denote by $C_P(\mu)$ the best such constant. 

To conclude the proof, we need a Poincaré inequality for the measure $d\tilde{\nu}(x) = \frac{1}{Z}\frac{\log^2(e+\abs{x})}{\cosh(\pi x/2)}$, where $Z$ is a normalizing factor. This is the content of the next lemma
\begin{lemma}\label{lem_poincare_perturb}
    Let $\nu = \rho(x)dx$ be a probability on $\RR$ with Poincaré constant $C_P(\nu) <\infty$, and let $\tilde{\nu}$ be the perturbated measure $\frac{d\tilde{\nu}}{d\nu}=\frac{1}{Z}\log^2(e+ \abs{x})$, then 
    $$C_P(\tilde{\nu}) \leq 4C_P(\nu)\left(1+ \frac{1}{2}\log^+\frac{4C_P(\nu)}{e}\right)^2.$$
    where for any $x>0, \ \log^+x = \max(0,\log x)$
\end{lemma}
\begin{proof}
Let $f\in L^2(\tilde{\nu})$ such that $f'\in L^2(\tilde{\nu})$, let $g(x) = f(x)\log(e+\abs{x})$ and assume that $\int g d{\nu} =0$. By the Poincaré inequality for $\nu$ :

\begin{align*}
    \Var_{\nu}(g) &= \int_{\RR}f^2\log^2(e+\abs{x}) d\nu \\ &\leq \ C_P(\nu) \int_{\RR} \left(f'\log(e+\abs{x}) + \frac{f}{e+\abs{x}})\right)^2\ \ d\nu \\
    &\leq 2ZC_P(\nu) \ \int_{\RR} (f')^2 d\tilde{\nu} \ + \ 2C_P(\nu)\int_\RR \frac{f^2}{(e+\abs{x})^2} d \nu
\end{align*}
So that,
\begin{equation}\label{eq_435}
    \frac{1}{Z}\int_{\RR} f^2(\log^2(e+\abs{x})-\frac{2C_P(\nu)}{(e+\abs{x})^2})d\nu \leq 2C_P(\nu) \ \int_{\RR} (f')^2 d\tilde{\nu}. 
\end{equation}
Now if $C_P(\nu) \leq \frac{e}{4}$, we immediately get 
$$C_P(\tilde{\nu}) \leq 4C_P(\nu).$$
On the other hand, if $C_P(\nu)>\frac{e}{4}$, we set $\lambda = 2\sqrt{\frac{C_P(\nu)}{e}}>1$, and we denote by $\nu_\lambda$ and $\tilde{\nu}_\lambda$ the dilated measures defined by
$$\nu_\lambda(A) = \nu(\lambda A) \ ; \quad \tilde{\nu}_\lambda(A) = \tilde{\nu}(\lambda A)$$
for all measurables sets $A$. We have, by scaling of the Poincaré constant, $C_P(\nu_\lambda) = C_P(\nu)/\lambda^2 = e/4$. Thus by what precedes, for any regular enough $f$ with
$$\int_\RR f\log^2(e+\abs{x})\nu_\lambda(dx) = 0 $$
we have
\begin{equation}
    \int_\RR f^2\log^2(e+\abs{x})\nu_\lambda(dx) \leq e \int_\RR (f')^2\log^2(e+\abs{x})\nu_\lambda(dx).
\end{equation}
Since $\lambda>1$, we remark that for any $x>0$
$$\log^2(e+\lambda x) \ \leq \ \left(\log(e+\abs{x}) + \log \lambda\right)^2 \ \leq \ (1+\log(\lambda))^2\log^2(e+\abs{x}).$$
Notice that, precisely, for all $x\in\RR$
$$\frac{d\tilde{\nu}_\lambda}{d\nu_\lambda}(x) = \frac{d\tilde{\nu}}{d\nu}(\lambda x) = \frac{\log^2(e+\lambda\abs{x})}{Z}.$$
We now prove a Poincaré inequality for $\tilde{\nu}_\lambda$. Without loss of generality we assume that $$\int_\RR f\log^2(e+\abs{x})\nu_\lambda(dx) = 0$$
and we compute :
\begin{align*}
    \Var_{\tilde{\nu}_\lambda}(f) &\leq \int_\RR f^2 \ d\tilde{\nu}_{\lambda}(x) \\
    & = \frac{1}{Z} \int_\RR f^2\ \log^2(e+\lambda\abs{x}) d\nu_\lambda(x) \\
    & \leq \frac{(1+\log \lambda)^2}{Z}\int_\RR f^2\ \log^2(e+\abs{x}) d\nu_\lambda(x) \\
    &\leq \frac{e(1+\log \lambda)^2}{Z} \int_\RR (f')^2\ \log^2(e+\abs{x}) d\nu_\lambda(x) \\
    & = e(1+\log \lambda)^2 \int_\RR (f')^2 \ d\tilde{\nu}_\lambda(x)
\end{align*}
Thus
$$C_P(\tilde{\nu}_\lambda) \leq e(1+\log \lambda)^2.$$
Finally
$$C_P(\tilde{\nu}) = \lambda^2C_P(\tilde{\nu}_\lambda) \leq 4C_P(\nu)\left(1+ \frac{1}{2}\log\frac{4C_P(\nu)}{e}\right)^2.$$
\end{proof}

We proved Theorem \ref{thm:main} when $f\in\mathcal{F}_2$, it remains to extend it to all square integrable functions. Let $f\in L^2(\nu)$ such that $\int (\log^2(e+\abs{x})(f')^2 d\nu < \infty$. By density of the polynomials in $L^2(\tilde{\nu})$, which admits exponential moments, there exists some polynomial $Q$ such that 
$$\int_\RR (f'-Q')^2\log^2(e+\abs{x})d\nu \leq \eps$$
Thus, by adding an appropriate constant to $Q$ and using the Poincar\'e inequality (Lemma \ref{lem_poincare_perturb}), we also have
$$\int_\RR (f-Q)^2\log^2(e+\abs{x})d\nu \leq C\eps$$
for some constant $C>0$.

We define 
$$T :  \RR[X] \mapsto  L^2(\nu)$$
by 
$$T(Q) = \sum_{k\geq 1}\log(e+k)\langle Q,P_k\rangle_{L^2(\nu)} P_k$$
where $P_k$ are the MP polynomials and $\RR[X]$ is equipped with the Hilbert norm
$$\norm{Q}^2 = \int Q^2\log^2(e+\abs{x}) \ d\nu \ + \ \int (Q')^2\ d\nu.$$
$T$ is uniformly bounded on $\RR[X]$ which is a dense subset of $L^2(\nu)$. Thus $T$ extends to a bounded operator on this Hilbert space, with the same norm, which is what we needed.

\subsection{Tightness of the bound.}\label{subsec_tightness}
In this section we prove the tightness of Theorem \ref{thm:main}. The computations at the end of Section \ref{section_preliminaries} suggest that already gaussians  provide a lower bound.

Set $F(x) = e^{-x^2/2}$ and as before, for any $\lambda \geq 1$, 
$$F_\lambda(x) = \frac{1}{\sqrt{\lambda}}F(\lambda x)$$
and remark that 
$$\int_\RR \abs{F'_\lambda(x)}^2\log^2{(e+\abs{x})}e^{-\abs{x}} \ dx \leq\int \abs{F'(\lambda x)}^2 \lambda \ dx = \sqrt{2\pi} $$
where we used that $\log^2(e+\abs{x})e^{-x} \leq 1$ for all non-negative $x$. Furthermore $F_\lambda \in \cF_2$ for all $\lambda\geq1$ as it is holomorphic and bounded on any horizontal strip.

Now let $(a_k)_{k\geq1}$ be a positive sequence with $\lim_{k\to\infty}a_k = +\infty$. We will prove that
\begin{equation}\label{eq_462}
    \sum_{k\geq 1} a_k \log^2(e+k)(F_\lambda)_k^2  \quad \underset{\lambda \to \infty}{\longrightarrow} \quad +\infty.
\end{equation}
Up to adding a constant if necessary, we assume that for all $k\in\NN$
$$a_k\geq 1.$$
We begin with a technical lemma
\begin{lemma}\label{lem_define_tau}
    There exists some positive convex function $\tau$ such that 
    \begin{itemize}
        \item For any $k\geq 1 \quad a_k\log^2(e+k) \geq \tau(\log^2(e+k))$
        \item $\lim_{x\to\infty} \frac{\tau(x)}{x} = +\infty$
        \item For any $x>16$, $\frac{\tau'(x)}{\tau(x)} \leq \frac{1}{4\sqrt{x}}$
        \item For any $x>0$, $\tau(x)\leq 1+x^2$
    \end{itemize}
\end{lemma}
\begin{proof}
Let 
$$0=x_1<x_2<\ldots<x_i<\ldots$$
be an increasing sequence such that 
$$\lim_{i\to\infty}x_i = +\infty$$
to be determined precisely later. We define $\tau$ to be affine on each $(x_i,x_{i+1})$ for $i\geq1$. More precisely, we set 
$$\tau(x_1)=\tau(0)=0$$
and we define, inductively, for any $k\geq1$
\begin{equation}
    \tau(x) = i(x-x_i) + \tau(x_i) \qquad \text{if }x\in[x_i,x_{i+1}[
\end{equation}
By definition, $\tau$ is positive, and $\tau'$ is increasing, and diverging. Thus $\tau$ is convex and 
$$\lim_{x\to\infty} \frac{\tau(x)}{x} = +\infty.$$
Now we want to ensure that any $k\geq 1$,  $$a_k\log^2(e+k) \geq \tau(\log^2(e+k))$$
For all integers $i\geq1$ we set 
$$k_i=\inf\{k\in\NN, \ a_k\geq i\}$$
which is finite since $a_k$ diverge to infinity. Notice also that $k_1 = 0$. For any $i\geq 1$, we choose $x_{i+1}$ such that 
\begin{equation}\label{eq810}
    x_{i+1} \geq \log^2(e+k_i).    
\end{equation}
Then, for any $k\geq1$, there exists $i\geq1$ such that 
$$k_i\leq k \leq k{i+1}$$
Thus, by the definition of $k_i$,
$$a_k\log^2(e+k) \geq i\log^2(e+k).$$
On the other hand, by \eqref{eq810}
$$\log^2(e+k)\leq \log^2(e+k_i)\leq x_{i+1}.$$
Now, notice that for any $x\leq x_{i+1}$,
$$\tau(x) = \int_{0}^{x}\tau'(t)\ dt \leq ix$$
Thus
$$\tau(\log^2(e+k))\leq i\log^2(e+k)\leq a_k\log^2(e+k)$$
which is what we wanted. It remains to ensure the growth conditions on $\tau'$ and $\tau$.

Let $x>0$ and let $i\geq1$ such that $x\in(x_i,x_{i+1})$. We want to ensure that
$$\frac{\tau'(x)}{\tau(x)} = \frac{i}{i(x-x_i)+\tau(x_i)}\leq \frac{1}{4\sqrt{x}}$$
or equivalently,
$$i(x-x_i) +\tau(x_i) \geq 4i\sqrt{x}.$$
First we assume that $i\geq 2$. Clearly, it is enough to the inequality at at $x=x_i$. This amounts to 
$$\tau(x_i) \geq 4i\sqrt{x_i}.$$
But we always have $\tau(x)\geq x$, thus it is enough to assume that for all $i\geq2$
\begin{equation}\label{eq832}
    x_i\geq 16i^2.
\end{equation}
When $i=1$, the inequality becomes 
$$\frac{1}{x} \leq \frac{1}{4\sqrt{x}}$$
which is valid for $x\geq16$. We have verified that for all $x\geq \min(x_2,16)$
$$\frac{\tau'}{\tau} \leq \frac{1}{4\sqrt{x}}$$
Finally, for the growth condition on $\tau$, by what precedes we chose $x_2\geq64\geq1$. In particular, 
$$\tau(1)=1.$$
Let $x\geq 1$ and let $i$ be such that $x_i\leq x \leq x_{i+1}$. If $i=1$,
$$\tau'(x) = 1 \leq x.$$
If $i\geq2$, using \eqref{eq832},
$$\tau'(x) = i \leq x_i \leq x.$$
Thus,
$$\tau(x) \leq 1 + \int_1^x t \ dt \leq 1 + t^2$$
\end{proof}
We now set $\varphi(\eps) = \tau\left(\log^2 \eps\right)$. We shall only use a weak consequence of \eqref{eq_260} :  
\begin{align}
        \sum \Gamma_\varphi(k) (F_\lambda)_k^2 &\gtrsim \int_{-\pi/8}^{\pi/8} \int_{v\in\RR}\abs{\widehat{K_u\Delta_{F_\lambda}}}^2 \ \Phi\left(\frac{1}{4}e^{-2\abs{v}}\right)\nonumber\\ 
        &\gtrsim \int_{-\pi/8}^{\pi/8} \int_{v\in\RR}\abs{\widehat{K_u\Delta_{F_\lambda}}}^2 \ \Phi\left(e^{-2\abs{v}}\right)\label{eq_489}.
\end{align}
Where used that $\Phi(x/4) \geq \Phi(x)/4$, since $\varphi$ is decreasing. We start by evaluating $\Phi$ and $\Gamma_\varphi$
\begin{lemma}\label{lem_491} For any $0\leq x\leq e^{-4}$ we have 
$$x\varphi(x) \ \leq \ \Phi(x) \ \leq \ 2x\varphi(x)$$    
\end{lemma}
\begin{proof}
    Set $\Psi(x) = x\varphi(x) = x\tau\log^2(x)$. Remember that $\tau(x)\leq 1+x^2$ thus $$\Psi(0) = 0 = \Phi(0).$$ We compute 
    $$\Psi'(x) = \log^2x + 2\log(x) \tau'(\log^2 x).$$
    Since $\tau'\geq 0$, for any $0\leq x\leq 1$,
    $\Psi'(x) \leq \Phi(x).$
    On the other hand, if $x\leq e^{-4}$, then 
    $$\log^2(x)\geq 16,$$
    thus
    $$\tau'(\log^2(x)) \leq \frac{\tau(x)}{4\abs{\log(x)}}.$$
    We get that on $[0,e^{-4}]$
    $\Psi'(x) \geq \log^2(x) -\frac{1}{2}\log^2(x) \geq \frac{1}{2}\Phi'.$
    Integrating the inequalities yields the desired result.
\end{proof}
\begin{lemma}\label{lem_505} There exists a universal constant $C>0$ such that for any $k\geq1$,
$$\Gamma_\varphi(k) \leq C \tau(\log^2(e+k))$$
\end{lemma}
\begin{proof}
Remember that from Lemma \ref{lemma_gammacompute}: 
$$\Gamma_\varphi(k) \leq \varphi(1/k) + 2k\Phi(1/k) $$
Now, if $k\geq e^{4}$, by the previous Lemma,
$$\Phi(1/k) \leq 4\varphi(1/k)$$
thus
$$\Gamma_\varphi(k) \leq 5\varphi(1/k) \leq 5\log^2(e+k)$$
On the other hand, if $k\leq e^{4}$,
\begin{align*}
    \Gamma_\varphi(k) &\leq \varphi(1/k) + 2k\Phi(1/k) \\
    &\leq \log^2(k) + 2k\Phi(1)\\
    &\leq \log^2(k) + 2e^{4}\Phi(1)\\
    &\leq (2e^{4}\Phi(1)+1)\log^2(e+k).
\end{align*}
The lemma is proved with $C=\max(2e^{4}\Phi(1)+1, \ 5).$
\end{proof}
Using the left hand side of Lemma \ref{lem_491} we can lower bound the right hand side of \eqref{eq_489}.
\begin{align}
     \int_{-\pi/8}^{\pi/8} \int_{v\in\RR}\abs{\widehat{K_u\Delta_{F_\lambda}}}^2 \ \Phi\left(e^{-2\abs{v}}\right) \ du dv &\geq \int_{-\pi/8}^{\pi/8} \int_{v\in\RR} \abs{\widehat{K_u\Delta_{F_\lambda}}}^2 e^{-2\abs{v}}\tau(\log^2(e^{-2\abs{v}})) \ du dv \nonumber \\
    &\geq\int_{v\in\RR} \abs{\widehat{K\Delta_{F_\lambda}}}^2 e^{-2\abs{v}}\tau(v^2) \ dv\label{eq_512}
\end{align}
where in the last line we used the Jensen inequality and $K$ is defined, for any $x\in\RR$, as 
$$ K(x) = \int_{-\pi/8}^{\pi/8}K_u(x) = \frac{2\sinh(\pi x/8)}{\sinh(\pi x/2)}.$$
Now,
$$\widehat{K\Delta_{F_\lambda}} = \widehat{K}*\widehat{\Delta F_\lambda}$$
Furthermore, for any $v\in\RR$
\begin{equation}\label{eq_516}
    \widehat{\Delta F_\lambda}(v) = \widehat{F_\lambda}(v)\sinh(v) = \frac{1}{\lambda^{3/2}}\hat{F}(v/\lambda)\sinh(v) = \frac{\sqrt{2\pi}}{\lambda^{3/2}}e^{-v^2/2\lambda^2}\sinh(v).
\end{equation}
A standard computation, which we detail in the appendix, shows that
\begin{equation}\label{eq_678}
    \hat{K}(v) = \frac{4}{1+\sqrt{2}\cosh(2v)}.
\end{equation}
We see that $\hat{K}(v)$ is a nice positive kernel which allows us to lower bound the convolution 
$$\widehat{K\Delta_{F_\lambda}} = \widehat{K}*\widehat{\Delta F_\lambda}.$$
Let $\lambda\geq1$. For any $v\geq 2$, 
\begin{align*}
\frac{\lambda^{3/2}}{4\sqrt{2\pi}} \hat{K} & * \widehat{\Delta F_{\lambda}} (v)    = \int_\RR e^{-t^2/2\lambda^2}\sinh(t)\frac{1}{1+\sqrt{2}\cosh(2(v-t))} \ dt\\
    &=\int_{\RR^+} e^{-t^2/2\lambda^2}\sinh(t)\frac{1}{1+\sqrt{2}\cosh(2(v-t))} \ dt + \int_{\RR^-} e^{-t^2/2\lambda^2}\sinh(t)\frac{1}{1+\sqrt{2}\cosh(2(v-t))} \ dt \\
    & = \int_{\RR^+} e^{-t^2/2\lambda^2}\sinh(t)\left(\frac{1}{1+\sqrt{2}\cosh(2(v-t))} - \frac{1}{1+\sqrt{2}\cosh(v+t)}\right) \ dt \\
    & \gtrsim\int_{\RR^+} e^{-t^2/2\lambda^2}\sinh(t)\frac{1}{1+\sqrt{2}\cosh(2(v-t))} \ dt \\
    &\geq\int_{v-1}^{v}e^{-t^2/2\lambda^2}\sinh(t)\frac{1}{1+\sqrt{2}\cosh(2(v-t))} \ dt \\
    &\gtrsim \sinh(v-1)e^{-v^2/2\lambda^2}\\
    &\geq e^{-v^2/2\lambda^2}e^{v}
\end{align*}
where we used that, for any $v\geq 2$ and $t\geq0$
\begin{align*}
    \frac{1+\sqrt{2}\cosh(2(v+t))}{1+\sqrt{2}\cosh(2(v-t))} \geq 1+c_0>1
\end{align*}
 and in the last line that $\sinh(v-1)\geq e^{v}/4$ for $v\geq 2$. By symmetry, we get that for all $\abs{v}\geq 2$,
$$\abs{\widehat{K_u\Delta F_\lambda}(v)} \geq \frac{c}{\lambda^{3/2}} e^{-v^2/2\lambda^2}e^{\abs{v}}$$
for some constant $c>0.$ Plugging this into \eqref{eq_489} and using \eqref{eq_512} as well as the estimate on $\Gamma_\varphi$ from Lemma \ref{lem_505}, we get 
\begin{align*}
    \sum_{k\geq 1}a_k\log^2(e+k)(F_\lambda)_k^2 &\geq \sum_{k\geq 1}\tau\left(\log^2(e+k)\right)(F_\lambda)_k^2 \\
    & \geq \frac{1}{C} \sum_{k=2}^{+\infty} \Gamma_\varphi(k)(F_\lambda)_k^2 \\
    & \gtrsim \int_{-\pi/8}^{\pi/8} \int_\RR \abs{\widehat{K_u\Delta_{F_\lambda}}}^2 e^{-2\abs{v}}\tau(v^2) \ dudv\\
    & \gtrsim \frac{1}{\lambda^3}\int_{\abs{v}\,\geq\, 2} e^{-v^2/\lambda^2}\tau(v^2)  \ dv \\
    & = \frac{1}{\lambda^2}\int_{\abs{w}\geq \frac{2}{\lambda}} e^{-w^2}\tau(\lambda^2w^2) \ dw\\
    &\geq \frac{Z}{\lambda^2}\int_{\abs{w}\geq 2} \frac{e^{-w^2}}{Z}\tau(\lambda^2w^2)\ dw\\
    &\geq \frac{Z}{\lambda^2} \tau\left(\lambda^2\int_{\abs{w}>2}w^2\frac{e^{-w^2}}{Z} \ dw\right) \\
    &\gtrsim \frac{\tau(C_1\lambda^2)}{\lambda^2} \underset{\lambda \to +\infty}{\longrightarrow} \quad +\infty,
\end{align*}
where $Z = \int_{\abs{w}\geq 2}e^{-w^2} >0$ and $C_1 = \frac{1}{Z}\int_{\abs{w}\geq 2}w^2e^{-w^2}>0 $ are two positive constants and we used Jensen inequality, as $\tau$ is convex.

Thus \eqref{eq_462} is proved. That is we showed that there exists a family of functions $F_\lambda$ such that 
$$\frac{1}{\sqrt{2\pi}}\int (F'_\lambda)^2(x) \ dx \leq 1, $$
but
\begin{equation*}
    \sum_{k\geq 1} a_k \log^2(e+k)(F_\lambda)_k^2  \quad \underset{\lambda \to \infty}{\longrightarrow} \quad +\infty.
\end{equation*}
By the Banach-Steinhaus theorem, there exists a function $F_\infty$ with
$$\int (F'_\infty)^2(x) \ dx \leq 1,$$
but
$$\sum_{k\geq 1} a_k \log^2(e+k)(F_\infty)_k^2 = +\infty.$$
It is a bit stronger than what was announced in Theorem \ref{thm:main}.

\appendix

\section{Appendix: Inequalities related to hyperbolic functions}
\subsection{Inequality 1}
We prove that
$$\frac{\cosh(x)-1}{x^2} \leq \frac{2\cosh(x)}{1+x^2}.$$
We work with $x\geq0$. If $x\geq 1$ there is nothing to prove as $1+x^2\leq 2x^2$. For $x\leq1$, we claim that
$$\frac{\cosh(x)-1}{x^2} \leq \frac{\cosh(x)}{1+x^2}.$$
Indeed, this is equivalent to 
$$\cosh(x) \leq 1+x^2,$$
which is true when $x\leq1$. Both functions are equal in $0$ and the derivative of the difference  satisfy :
$$\sinh(x)-2x \leq \sinh(1)x-2x \leq 0.$$
\subsection{Inequality 2}
We prove that 
\[
\frac{1}{2(1 + |x|)} \leq \frac{\cosh(x) - 1}{x \sinh(x)} \leq 1/2.
\]
Let $$F(x) = \frac{\cosh(x) - 1}{x\sinh(x)}.$$
Then for $x\geq0$,
\begin{align*}
    x\sinh^2(x)F'(x) &= x\sinh^2(x) - \sinh(x)\cosh(x)+\sinh(x) -x\cosh^2(x) + x\cosh(x)\\
    &= \sinh(x)+x\cosh(x)-\sinh(x)\cosh(x)-x\\
    &=x(\cosh(x)-1)-\sinh(x)(\cosh(x)-1)\leq0.
\end{align*}
Thus
$$F(x) \leq F(0) = \frac{1}{2}.$$
We prove the lower bound, for $x\geq0$. Let
$$G(x) = \cosh(x)-1 - \frac{x\sinh(x)}{2(1+x)}.$$
Then,
$$G'(x) = \sinh(x) - \frac{\sinh(x) + x(x+1)\cosh(x)}{2(1+x)^2}.$$
Thus, $G'\geq 0$ if and only if 
\begin{align*}
    \sinh(x)\left[2+4x+2x^2-1\right] \geq x(x+1)\cosh(x).
\end{align*}
Using that $\cosh(x) = e^{-x}+\sinh(x) \leq 1+ \sinh(x) x$, it is enough to show that
\begin{align*}
    &\sinh(x)\left[1+4x+2x^2\right] \geq x(x+1)(1+\sinh(x)) \\
    \Longleftrightarrow &\sinh(x)\left[1+3x+x^2\right] \geq x(x+1)
\end{align*}
which finishes the proof, since 
$$\sinh(x) \geq x.$$

\section{Appendix: Fourier Transform}
Instead of computing the Fourier transform of $K$, we equivalently compute the Fourier transform of 
$$\frac{1}{\sqrt{2}+\cosh(x)}.$$ 
Set
\[
I(v)=\int_{-\infty}^{\infty}\frac{e^{i x v}}{\sqrt{2}\cosh x+1}\,dx.
\]
$I(v)=I(-v)$ so we assume that $v>0$.
Define
\[
f(z)=\frac{e^{ivz}}{\sqrt{2}\cosh z+1}.
\]
We integrate $f$ on a contour which is a rectangle with edges
$$[-R,R] \quad \text{and} \quad [-R+2i\pi,R+2i\pi]$$
and take $R\to\infty$. Notice that $\cosh(z+2\pi i)=\cosh z$ so that 
$$f(x+2\pi i)=\frac{e^{iv(x+2i\pi)}}{\sqrt{2}\cosh x+1} = e^{-2v}f(x)$$
Thus integrating on the contour mentioned before and taking the limit as $R\to\infty,$ we find
\[(1-e^{-2\pi v})I(v) = 2i\pi\sum_{z} Res(f,z) .
\]
where the sum is taken over all poles $z$ inside the strip
\[
0\leq Im(z)\leq 2\pi.
\]
The poles are for 
\[\cosh z = -\frac{1}{\sqrt{2}}.
\]
That is 
\[
\cosh a \cos b + i\,\sinh a\sin b = -\frac{1}{\sqrt{2}}.
\]
This forces $a=0$, so that we need to solve for
\[
 \cos b = \frac{-1}{\sqrt{2}}.
\]
when $0\leq b \leq 2\pi$. Thus there are two poles :
\[
z_1=i\frac{3\pi}{4} \quad \text{and} \quad z_2=i\frac{5\pi}{4}.
\]
The residue at pole $z_k$, $k=1,2$ is given by 
\[
Res(f,z_k)=\frac{e^{ivz_k}}{g'(z_k)},
\]
where
\[
g(z)=\sqrt{2}\cosh z+1.\]
Thus
\[g'(z_k)=\sqrt{2}\sinh z_k = \sqrt{2}\sin(b_k).
\]
where $z_k = ib_k$. We find :
\begin{itemize}
    \item At $z_1=i\frac{3\pi}{4}$

    \[
\sinh(i\frac{3\pi}{4})= i\sin(\frac{3\pi}{4})
= i\frac{\sqrt{2}}{2}.
\]
Also,
\[
e^{i v z_1}= e^{i v\,(i\,\frac{3\pi}{4})} = e^{-3\pi v/4}.
\]
Therefore, the residue is
\[
Res(f,z_1)=\frac{e^{-3\pi v/4}}{\sqrt{2}\sin(z_1)} = \frac{e^{-3\pi v/4}}{i}.
\]
\item At $z_2=i\frac{5\pi}{4}$

Similarly we find
\[
Res(f,z_2)=-\frac{e^{-5\pi v/4}}{i}.
\]
\end{itemize}
By the residue theorem,
\begin{align*}
    (1-e^{-2\pi v})I(v) = 2\pi(e^{-3\pi v/4}-e^{5\pi v/4}).
\end{align*}
Thus
\begin{align*}
    I(v) &= \frac{2\pi(e^{-3\pi v/4}-e^{5\pi v/4})}{1-e^{-2\pi v}}.\\
\end{align*}
Now notice that 
\[e^{-3\pi v/4}- e^{-5\pi v/4} = e^{-\pi v}\, (e^{\pi v/4}- e^{-\pi v/4}) = 2\, e^{-\pi v}\,\sinh(\frac{\pi v}{4}),\]
and 
\[1-e^{-2\pi v} = 2e^{-\pi v}\sinh(\pi v).\]
We conclude that
\[I(v) = \frac{2\pi\sinh(\pi v/4)}{\sinh(\pi v)}.\]
The claim that 
$$\mathcal{F}\left(\frac{2\sinh(\pi x/8)}{\sinh(\pi x/2)}\right) = \frac{4}{1+\sqrt{2}\cosh(2v)}$$
is equivalent to
$$\mathcal{F}\left(\frac{4}{1+\sqrt{2}\cosh(2v)}\right) = \frac{4\pi\sinh(\pi x/8)}{\sinh(\pi x/2)}.$$
By what precedes,
$$\mathcal{F}\left(\frac{4}{1+\sqrt{2}\cosh(2v)}\right) = 4\frac{I(v/2)}{2} = 2I(v/2) = \frac{4\pi \sinh(\pi v/8)}{\sinh (\pi v /2)}$$
which is what we wanted.

\bibliographystyle{alpha}
\bibliography{biblio}

\end{document}